\documentclass[a4paper]{amsart}
\usepackage{graphicx}
\usepackage{graphics}
\usepackage{amsmath}
\usepackage{latexsym}
\usepackage{amssymb}
\usepackage{subfigure}
\usepackage{psfrag}
\usepackage[all]{xy}
\xyoption{matrix}
\xyoption{arrow}
\begin{document}

\newcommand{\add}{\operatorname{add}\nolimits}
\newcommand{\ind}{\operatorname{ind}\nolimits}
\newcommand{\Hom}{\operatorname{Hom}\nolimits}
\newcommand{\Ext}{\operatorname{Ext}\nolimits}
\newcommand{\C}{\operatorname{\mathcal C}\nolimits}
\newcommand{\Z}{\operatorname{\mathbb Z}\nolimits}
\newcommand{\R}{\operatorname{\mathbb R}\nolimits}

\newcommand{\Tn}{{\mathcal T}_n}
\newcommand{\T}{\operatorname{{\mathcal T}}\nolimits}
\newcommand{\Tfive}{{\mathcal T}_5}
\newcommand{\Tinf}{{\mathcal T}_{\infty}}
\newcommand{\CTn}{\operatorname{\mathcal C}_{{\mathcal T}_n}\nolimits}
\newcommand{\CTinf}{\operatorname{\mathcal C}_{{\mathcal T}_\infty}\nolimits}
\newcommand{\D}{\operatorname{\rm D}\nolimits}
\newcommand{\GT}{\Gamma_{\mathcal T_n}}
\newcommand{\GA}{\Gamma(\An)}
\renewcommand{\P}{\operatorname{\mathbb P}\nolimits}

\newcommand{\An}{\mathbb{A}(n)}
\newcommand{\A}{\operatorname{{\mathbb{A}}}\nolimits}
\newcommand{\U}{\mathbb{U}}
\newcommand{\Cn}{\mathrm{Cyl}(n)}

\renewcommand{\r}{\operatorname{\underline{r}}\nolimits}

\newtheorem{lemma}{Lemma}[section]
\newtheorem{prop}[lemma]{Proposition}
\newtheorem{corollary}[lemma]{Corollary}
\newtheorem{theorem}[lemma]{Theorem}
\newtheorem{example}[lemma]{Example}

\theoremstyle{definition}
\newtheorem{definition}[lemma]{Definition}
\newtheorem{remark}[lemma]{Remark}

\title[A geometric model of tube categories]{A geometric model of tube categories}

\author[Baur]{Karin Baur}
\address{
Institut f\"ur Mathematik und wissenschaftliches Rechnen \\
Universit\"at Graz \\
Heinrichstrasse 36 \\
A-8010 Graz \\
Austria
}
\email{baurk@uni-graz.at}

\author[Marsh]{Robert J. Marsh}
\address{School of Mathematics \\
University of Leeds \\
Leeds LS2 9JT \\
England
}
\email{marsh@maths.leeds.ac.uk}


\begin{abstract}
We give a geometric model for a tube category in terms of homotopy classes of oriented arcs in an annulus with marked points on its boundary. In particular, we interpret the dimensions of extension groups of degree $1$ between indecomposable objects in terms of negative geometric intersection numbers between corresponding arcs, giving a geometric interpretation of the description of an extension group in the cluster category
of a tube as a symmetrized version of the extension group in the tube. We show that a similar result holds for finite dimensional representations of the linearly oriented quiver of type $A_{\infty}^{\infty}$.
\end{abstract}

\date{24 April 2012}


\maketitle

\section{Introduction}

Let $k$ be an algebraically closed field. A \textit{tube category} $\Tn$ of rank $n$ over $k$ can be defined as the category of finite dimensional nilpotent representations of an oriented $n$-cycle over $k$. We introduce a geometric model of such a category using an annulus with $n$ marked points on the outer
boundary and none on the inner boundary. The indecomposable objects of $\Tn$ are in bijection with certain oriented arcs in the annulus, i.e.\ those which are oriented anticlockwise and which are not homotopic to a boundary arc between successive marked points; we refer to them as admissible arcs.
The AR-quiver of the tube can be realised via the geometry of the arcs. We then show that:

\begin{theorem} \label{t:mainresult}
Let $\alpha$ and $\beta$ be admissible arcs in an annulus with $n$ marked points on its outer boundary, and let $A$ and $B$ be
the corresponding indecomposable objects in a tube $\Tn$ of rank $n$. Then the dimension of $\Ext^1_{\Tn}(A,B)$ coincides with the negative geometric intersection number of the pair $\alpha,\beta$.
\end{theorem}

Cluster categories were introduced in~\cite{bmrrt} (and in~\cite{ccs} for
type $A$) in order to model the cluster algebras~\cite{fz1} of Fomin-Zelevinsky. They have since been used widely in representation theory
(see e.g.~\cite{k2} for a survey).
Since a tube is skeletally small it follows from Keller's theorem~\cite[\S9.9]{k1} that
the cluster category $\CTn$ associated to a tube (which we call a cluster tube) is triangulated (see also~\cite[\S5]{bkl1} and~\cite[\S2]{bmv}).
We show that the geometric intersection number of $\alpha,\beta$ is equal
to the sum of the negative geometric intersection number and the positive geometric intersection
number.
This total number coincides with the dimension of $\Ext^1_{\CTn}(A,B)$.
We thus obtain a geometric interpretation of the fact that
$$\Ext_{\CTn}(A,B)\cong \Ext_{\Tn}(A,B)\oplus D\Ext_{\Tn}(B,A),$$
(which can be proved as in~\cite[Prop. 1.7]{bmrrt})
for indecomposable objects $A$ and $B$ in $\Tn$, where $D=\Hom_k(-,k)$.
We also show that similar results hold for the category of finite dimensional representations of the linear orientation of the quiver $A_{\infty}^{\infty}$, a hereditary category whose AR-quiver is of type $\mathbb{Z}A_{\infty}$.

We note that geometric models of cluster categories and of the generalised
cluster categories of Amiot~\cite{a} use unoriented arcs in marked
surfaces: see~\cite{bz,ccs,hj,sc}.
Such models were used for cluster algebras in~\cite{fz2,fst}. Our perspective here is that by using oriented arcs we can model the classical categories $\Tn$ arising in representation theory. The $2$-Calabi-Yau $\Ext^1$-symmetry of the cluster category of $\Tn$ then arises by forgetting the orientation on the arcs. We note that oriented arcs have also been used in~\cite{bm} to obtain a geometric model for the root category~\cite{h} of type $A$ (the quotient of the bounded derived category by the square of the shift).

We remark that~\cite{g} and~\cite[3.18]{bz} give a construction of the AR-quiver of a cluster tube (in the latter case, as a subcategory of a generalised cluster category in a more general situation)
using unoriented arcs in an annulus. It follows from~\cite[5.3]{bz} that if the geometric intersection number of two non-homotopic admissible arcs in the annulus is non-zero then the $\Ext^1$-group between the corresponding objects in the cluster tube is non-zero, and it follows from~\cite[5.1]{bz} that the geometric intersection number of an arc with itself vanishes (see Remark~\ref{rem:selfintersection}) if and only if the $\Ext^1$-group of the corresponding object with itself in the cluster tube vanishes. Thus, in the case of the cluster tube, we obtain a strengthening of these results.

In Section $2$ we recall tube categories and indicate how the AR-quiver of
the tube can be constructed geometrically using admissible arcs in the annulus. In Section $3$ we compute geometric intersection numbers between admissible arcs in the annulus by lifting them to the universal covering space, and prove the main result.
In Section $4$ we give an interpretation of our model using doubled arcs which,
in particular, gives a connection with a model for the rigid part of the
cluster tube given in~\cite[\S3]{bmv}. In Section 5 we consider the
case of the category of finite dimensional representations of the linear orientation of $A_{\infty}^{\infty}$.

Since completing this work, we learnt of interesting independent work in
the diploma thesis~\cite{w} which, in particular, gives a bijection
between string modules over a quiver of type $\widetilde{A_n}$
(i.e. a quiver whose underlying graph is a cycle, but is not an oriented
cycle itself) and certain oriented arcs in the annulus (regarded as a
cylinder). A geometric interpretation of the AR-quiver is given and it is
shown that the negative geometric intersection number of two arcs coincides with the
dimension of the extension group of the corresponding modules.
Since tubes can be obtained as subcategories of the module categories of such
quivers, this gives an alternative proof of Theorem~\ref{t:mainresult}.

Our approach includes a topological explanation of how geometric intersection numbers in
the annulus and its universal cover are related (see Section 3). In addition, it
links up interpretations of geometric intersection numbers with and without
orientation, thus strengthening results of~\cite{bz} in this case.
We also give an explanation for the doubled arc model for exceptional objects
in the tube appearing in~\cite{bmv} and generalise to the
$\mathbb{Z}A_{\infty}$ case using the universal covering of the annulus.

\section{The AR-quiver of a tube}
\subsection{Categories of tube shape} \label{ssec:tubes}

We refer to~\cite{ars} for background on the theory of finite dimensional
algebras and their representations.
For more information on tubes, see~\cite{ri,ss}.
A stable translation quiver $\Gamma$ without multiple arrows is said to be of
{\em tube type} if it is of the form $\Gamma=\Z A_{\infty}/n$ for some $n>0$, cf.~\cite[\S3]{ri}.
In that case we say that the tube $\Gamma$ has rank $n$ and
we call the additive hull of the mesh category of $\Gamma$ a
{\em tube of rank $n$}, denoting it $\Tn$. If $C_n$ is a cyclically oriented
quiver with $n$ vertices then the category of finite-dimensional nilpotent
representations of $kC_n$ is equivalent to $\Tn$
(this follows from~\cite[3.6(6)]{ri}; see also~\cite[III.1.1]{rv}).

The category $\Tn$ is Hom-finite, abelian, hereditary and
uniserial~\cite[\S 3]{ri}. We have the duality
$\Ext_{\Tn}^1(X,Y)\cong D\Hom_{\Tn}(Y,\tau X)$ in $\Tn$,
where $\tau$ denotes the AR-translate.
We can form its
cluster category as in~\cite[\S1]{bmrrt}, obtaining
the {\em cluster tube} $\CTn$ of rank $n$ as the orbit category
$\D^b(\Tn)/\tau^{-1}[1]$.
As noted above, Keller's theorem~\cite[\S9.9]{k1} applies,
so the cluster tube is triangulated. This category has also been
studied in~\cite{bkl1,bkl2,bmv}.
Its AR-quiver is the same as the AR-quiver of
$\Tn$ (arguing as in~\cite[Prop.\ 1.3]{bmrrt}). To abbreviate we write
$\GT:=\Gamma_{\rm AR}(\Tn)$ for the AR-quiver of $\Tn$.

\subsection{Arcs in the annulus and its universal cover}
\label{ssec:universalcover}

Let $\An=\mathbb{A}(n,0)$ be an annulus with $n$ marked points on the outer
boundary. Our goal is to use oriented arcs in $\An$ as a geometric model
for the tube and the cluster tube as introduced in Subsection~\ref{ssec:tubes}.

\begin{remark}
In~\cite[Example 6.9]{fst}, the authors consider annuli of the form $\mathbb{A}(r,s)$ with $r,s\geq 1$. These
have $r$ marked points on the outer boundary, $s$ marked points on the inner
boundary and unoriented arcs between these marked points.
More generally, the authors study triangulated surfaces as a geometric model
for cluster algebras. In particular, the annuli $\mathbb{A}(r,s)$ give rise to
cluster algebras associated to quivers of type $\widetilde{A}_{r+s-1}$ with
$r$ arrows in one direction and $s$ arrows in the other direction.
Thus it makes sense to use the annulus $\An$, with
$n$ marked points on one boundary component only, to model $\Tn$, whose objects
can be considered as (nilpotent) representations of a cyclic quiver with $n$ arrows all
in the same direction.
\end{remark}

We consider smooth oriented arcs in surfaces with boundary.
Recall that two such arcs $\alpha$, $\beta$ which start
at the same point and finish at the same point are said to be {\em homotopy
equivalent} if there is a homotopy between them fixing the starting
point and the finishing point. We write $\alpha\sim\beta$ to denote this.

Let us now consider oriented arcs between marked points of the boundary
of $\An$. We start by labelling the marked points of $\An$ with $0,1,\dots,n-1$,
in an anti-clockwise direction.
We will always take arcs up to homotopy fixing the starting and ending points.
For simplicity, we will assume that the marked points are equally spaced
around the outer boundary of $\An$.
An oriented arc from a marked point $i$ to a marked point $l$, for
$i,l\in\{0,1,\dots,n-1\}$,
is homotopically equivalent to the boundary
arc which starts at $i$, goes around the boundary $k$ times
(for some $k\ge 0$) and then ends in $l$; we shall denote this arc by
$i_kl$.

It will be most convenient to describe oriented arcs and their intersections using
the universal cover of $\An$.
To be precise, observe that we may identify the annulus $\An$ with a (bounded)
cylinder $\Cn$ of height $1$
with $n$ marked points on the lower boundary: we regard the cylinder as a rectangle in $\R^2$
with the vertical sides identified, cf.\ Figure~\ref{fig:ann-cyl}.
As in the annulus, we label the marked points on the lower boundary of $\Cn$
from left to right and assume that arcs are oriented anti-clockwise (when viewed from the top of the cylinder), i.e. left to right in the rectangle in $\R^2$.

In a second step, we move from the cylinder $\Cn$ to its universal covering
space $\U=(\U,\pi_n)$, with
$\U=\{(x,y)\in\R^2\mid 0\le y\le 1\}$, an infinite strip in the plane.
We adopt the orientation on $\U$ inherited from its embedding in the plane.
The covering map $\pi_n:\U\rightarrow \Cn$ is induced from wrapping $\U$
around $\Cn$, i.e.\ the effect of $\pi_n$ on $(x,y)\in\U$ is to take the first coordinate modulo $n$:
\[
\pi_n :\U \to \Cn\, ,\ (x,y)\ \longmapsto\ (x \text{ mod } n,y),
\]
so $(\lambda n,y)$ is mapped to $(0,y)$ for any $\lambda\in \mathbb{Z}$.
We choose an orientation on $\Cn$ such that $\pi_n$ is orientation-preserving.
Via the identification of $\Cn$ with $\An$, this induces an orientation on
$\An$. We will also write $\pi_n$ to indicate the covering map from $\U$ to
$\An$ directly.
We take the integers $(r,0)\in\U$, $r\in\Z$, as marked points on the lower
boundary of $\U$.

We define $\sigma:\U\to \U$ to be the translation which adds $n$ to the $x$-coordinate
of a point in $\U$, with inverse $\sigma^{-1}$ subtracting $n$ from the
$x$-coordinate of a point, so that $\pi_n(\sigma^r(u))=\pi_n(u)$
for all $u\in\U$ and for all $r\in\Z$. We write $G:=\langle\sigma\rangle$ for
the group generated by $\sigma$ which thus acts naturally on $\U$.

Observe that for any $s\in\R$, the image of
$\U_s=\{(x,y)\in \U\mid s\le x<s+n\}$ under $\pi_n$ is
the cylinder $\Cn$.

It is clear that the oriented arc in the cylinder $\Cn$ corresponding to $i_kl$
starts at a marked point $i$ and then wraps around the cylinder $k$ times before
ending at a marked point $l$. We now introduce an alternative notation for
arcs in $\U$ and $\An$. We call the arcs that we would like to focus on
\textit{admissible}.

\begin{definition}
Let $i,j$ be integers with $j>i+1$. An {\em admissible arc} in $\U$
is an oriented arc in $\U$ starting at the marked point $i=(i,0)$
and ending in $j=(j,0)$; we denote it by $[i,j]$.
An oriented arc in $\An$ is said to be \textit{admissible} if it is of the form
$\pi_n([i,j])$ where $[i,j]$ is admissible in $\U$.
\end{definition}

Note that arcs in $\An$ of the form $\pi_n([i,i+1])$, which are homotopic to the part
of the boundary between adjacent marked points, are not admissible, and nor
are arcs of the form $\pi_n([i,j])$ with $i\geq j$.

\begin{definition}
Let $[i,j]$ be an admissible arc in $\U$. Then we define the {\em combinatorial length} of $[i,j]$ to be the integer $j-i$.
\end{definition}

\begin{remark}
1) If $\alpha$ is an admissible arc in $\An$ with starting point $i$, $0\le i < n$,
then it has a unique lift $\widetilde{\alpha}$ in $\U$ which starts at $i=(i,0)$. We call
this lift the {\em canonical lift of $\alpha$}.
Observe that all other liftings of $\alpha$ can be obtained by iteratedly applying
$\sigma$ or $\sigma^{-1}$ to $\widetilde{\alpha}$.
So we get $\pi_n^{-1}(\alpha)=G\widetilde{\alpha}$.

2) Observe that the winding number of the arc
$\pi_n([i,j])$ around the inner boundary of $\An$ is given by
$k=\lfloor\frac{j-i}{n}\rfloor$.
\end{remark}

\begin{figure}
\includegraphics[scale=.4]{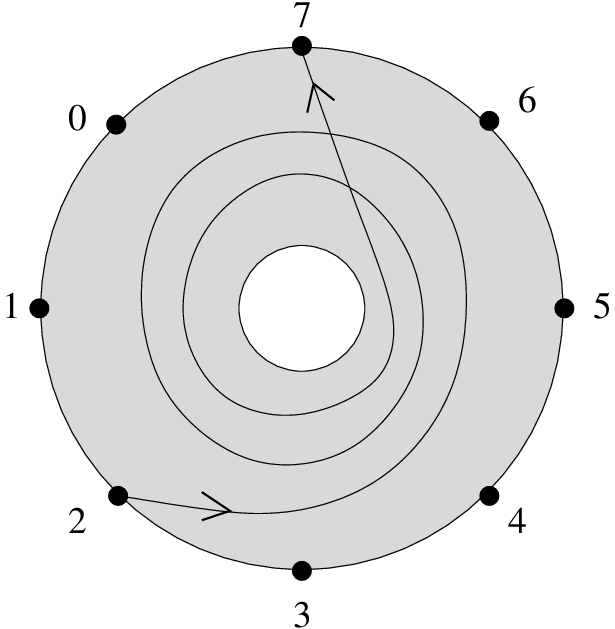}
\hspace{1.5cm}
\includegraphics[scale=.4]{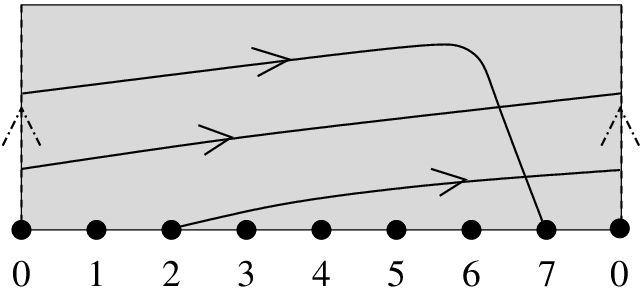}
\caption{$\alpha=\pi_8([2,23])$ in the annulus $\mathbb{A}(8)$ and in
the cylinder $\mathrm{Cyl}(8)$}\label{fig:ann-cyl}
\end{figure}

\begin{figure}
\includegraphics[scale=.4]{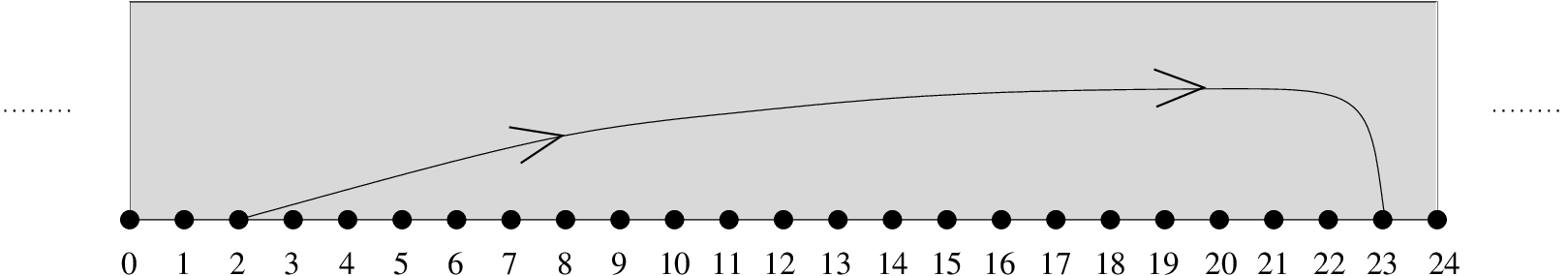}
\caption{The universal cover $\U$ with a lift of the arc $\alpha=\pi_8([2,23])$}\label{fig:univ}
\end{figure}

\subsection{A quiver of arcs} \label{ss:quiverofarcs}

We can associate a stable translation quiver to $\An$, called
$\GA=\Gamma_{\mathrm arc}(\An)$.
The vertices of $\GA$ are the admissible arcs in $\An$.
The arrows of $\GA$ are as follows.
Let $\alpha$ and $\beta$ be admissible arcs in $\An$.
Then there is an arrow from $\alpha$ to $\beta$ if and only if
there are admissible arcs $[a,b]$ and $[c,d]$ in $\U$ with
$\pi_n([a,b])=\alpha$, $\pi_n([c,d])=\beta$ and either $c=a+1$ and $d=b$,
or $c=a$ and $d=b+1$.

The translation $\tau$ on $\GA$ is induced by the map
$i\mapsto i-1$ (taken modulo $n$) for $i$ a marked point of $\An$.
This clearly defines a structure of stable translation quiver on the
set of admissible arcs in $\An$.
The case $n=5$ is shown in Figure~\ref{fig:arc-quiver} as an example.

\begin{figure}
$$\xymatrix@C=-0.1cm{
&&&&& \vdots &&&&& \\
\ar@{--}[r] & \pi_5[3,8] \ar[dr] \ar@{--}[rr] && \pi_5[4,9] \ar@{--}[rr] \ar[dr] && \pi_5[0,5] \ar@{--}[rr] \ar[dr]
&& \pi_5[1,6] \ar@{--}[rr] \ar[dr] &&
\pi_5[2,7] \ar@{--}[r] \ar[dr]&& \\
\pi_5[3,7] \ar[ur] \ar@{--}[rr] \ar[dr] && \pi_5[4,8] \ar@{--}[rr] \ar[ur] \ar[dr] && \pi_5[0,4] \ar@{--}[rr] \ar[ur] \ar[dr] && \pi_5[1,5]
\ar@{--}[rr] \ar[ur] \ar[dr] && \pi_5[2,6] \ar@{--}[rr] \ar[dr] \ar[ur] && \pi_5[3,7]
\\
\ar@{--}[r] & \pi_5[4,7] \ar@{--}[rr] \ar[ur] \ar[dr] && \pi_5[0,3] \ar@{--}[rr] \ar[ur] \ar[dr] && \pi_5[1,4]
\ar@{--}[rr] \ar[ur] \ar[dr] && \pi_5[2,5] \ar@{--}[rr] \ar[ur] \ar[dr] && \pi_5[3,6]
\ar@{--}[r] \ar[ur] \ar[dr] & \\
\pi_5[4,6] \ar@{.}[uuuu] \ar@{--}[rr] \ar[ur] && \pi_5[0,2] \ar@{--}[rr] \ar[ur] && \pi_5[1,3] \ar@{--}[rr] \ar[ur]  &&
\pi_5[2,4] \ar@{--}[rr] \ar[ur] && \pi_5[3,5] \ar@{--}[rr] \ar[ur] && \pi_5[4,6] \ar@{.}[uuuu] \\
}$$
\caption{The translation quiver $\Gamma(\A(5))$.}
\label{fig:arc-quiver}
\end{figure}
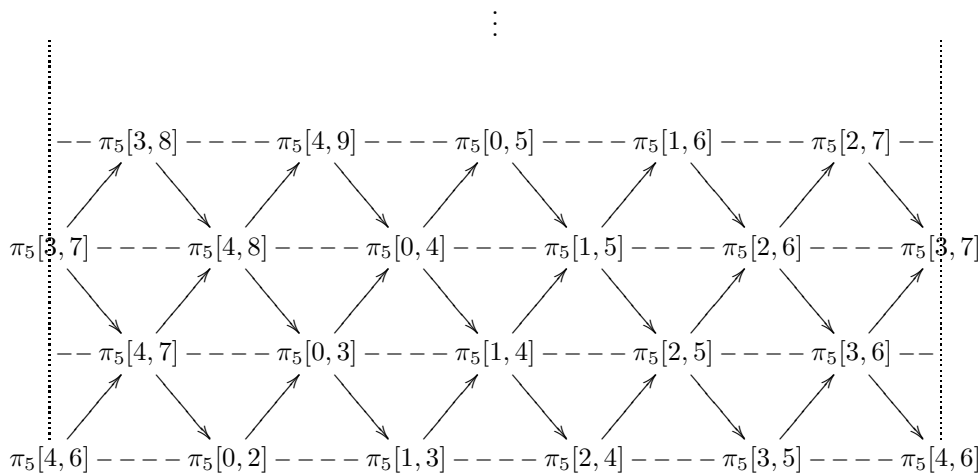
\subsection{Isomorphism of translation quivers}

Here we show that the quivers
$\GA$ and $\Gamma(\Tn)$ are isomorphic and thus
get the geometric model for tubes.
This isomorphism also appears in the Master's thesis~\cite{g} of
B. Gehrig and in~\cite[\S3.4]{bz} (in these
two cases, using unoriented arcs), and in~\cite[4.18]{w} (using oriented arcs).

To describe the AR-quiver $\Gamma(\Tn)$, recall that
$\Tn$ is uniserial. We denote the simple objects by $S_0$, $S_1$,
$\dots, S_{n-1}$.
Then, if $M$ is an indecomposable object in $\Tn$, it has
a unique composition series of the form $S_{i_1}, S_{i_2},\dots, S_{i_k}$,
for some $k\ge 1$,
where $i_1,i_2,\ldots ,i_k\in\{0,1,\dots,n-1\}$ and
$i_{j+1}=\overline{i_j+1}$ for all $j$,
where $\overline{a}$ denotes the reduction of $a$ modulo $n$.
If $a,b\in\mathbb{Z}$ and $a+1<b$, we denote the unique indecomposable object
of $\Tn$ with composition series $S_{\overline{a+1}},S_{\overline{a+2}},\ldots ,S_{\overline{b-1}}$ by $M[a,b]$. Note that every indecomposable object
arises uniquely in this way up to a shift of both entries by $n$, i.e.\ $M[a,b]\cong M[a+n,b+n]$; thus
we obtain a unique parametrization of the indecomposable objects of $\Tn$
(up to isomorphism) by restricting $a$ to be in the set $\{0,1,2,\ldots
,n-1\}$.

The AR-quiver of $\Tn$ is then easily described. It has vertices given
by the indecomposable objects $M[a,b]$ of $\Tn$, for $0\leq a \leq n-1$ and
$a+1<b$ (note that $M[a,a+2]\cong S_{\overline{a+1}}$,
for $a=0,1,\ldots ,n-1$).
For any such $a,b$, the inclusion $M[a,b]\rightarrow M[a,b+1]$
and the surjection $M[a,b]\rightarrow M[a+1,b]$ (if $a+2<b$)
correspond to arrows in the AR-quiver (using the above isomorphism if necessary); all arrows in the
AR-quiver arise in this way. The AR-translate $\tau$ takes $M[a,b]$ to
$M[a-1,b-1]$.

It is then straightforward to check that we have the following.

\begin{lemma} \label{l:ARiso}
The quivers $\GA$ and $\GT$ are isomorphic as stable translation quivers.
On vertices, the isomorphism is given by the map:
\[
\begin{array}{lccl}
\varphi: &  \{\mbox{admissible arcs in $\An$}\} & \longrightarrow & \ind(\Tn) \\
  & \pi_n([a,b]) & \longmapsto & M[a,b]
\end{array}
\]
where $a\in \{0,1,\dots,n-1\}$, $b\in\mathbb{Z}$ and $a+1<b$.
\end{lemma}

We remark that analogous results in the cluster category case have been obtained in
type A~\cite[Thm.\ 5.1]{ccs}, type D~\cite[\S5]{sc}, for a surface with marked points on the
boundary~\cite[\S1]{bz}, and in type $A_\infty$~\cite{hj} (see Section~\ref{s:Ainfinity}).
Note that here we are not considering the cluster category, but as we
observed in Section~\ref{ssec:tubes}, the AR-quivers of $\Tn$ and $\C_{\Tn}$
are isomorphic.
We remark that in~\cite{sc}, the puncture (and corresponding tagged arcs)
allows for the possibility of $3$ middle terms in the Auslander-Reiten
triangles; this does not happen here, where there are only $1$ or $2$ middle
terms in the Auslander-Reiten sequences.
For a geometric interpretation of all non-split short exact sequences
in $\Tn$, we refer to~\cite{bbm}.

We also note that by a result of Ringel~\cite[\S3]{ri} (see
also~\cite[X.2.6]{ss}), $\Tn$ is known to be standard,
i.e.\ the subcategory
$\ind(\Tn)$ is equivalent to the mesh category of $\Gamma_{\Tn}$.
Thus Lemma~\ref{l:ARiso} can be regarded as giving a geometric realisation
of the category $\Tn$ (i.e.\ as the mesh category of $\Gamma(\An)$), which is
analagous to~\cite[Thm.\ 5.2]{ccs},~\cite[Thm.\ 4.3]{sc}. As mentioned in
Section~\ref{ssec:tubes}, the AR-quiver of $\C_{\Tn}$ is isomorphic to
that of $\Tn$. It follows that a $\C_{\Tn}$ is not standard (as it is not
equivalent to $\Tn$) and we do not obtain a similar description of the
cluster tube.

\section{Geometric intersection numbers as dimensions of extension groups}
\label{s:intersectionnumbers}

In this section we show that negative geometric intersection numbers between admissible
arcs in $\An$ can be interpreted as dimensions of extension groups between the
corresponding objects in $\Tn$ via the isomorphism in Lemma~\ref{l:ARiso}.
To do this we need to compute these geometric intersection numbers; we do it by showing
that they can be computed in the universal cover $\U$ of $\An$.
We were unable to find the exact statement we needed in the literature so
include here a full proof.

We have the following (see e.g.~\cite[5.4]{c}).
\begin{theorem} (Monodromy Theorem) \label{t:monodromy}
Let $\pi:E\rightarrow X$ be a covering space of a surface $X$ and let
$\gamma,\delta$ be paths in $E$ with common starting point.
Then $\gamma\sim\delta$ if and only if
$\pi(\gamma)\sim\pi(\delta)$.
It follows that, in this situation, $\gamma$ and $\delta$
have a common finishing point.
\end{theorem}

Two arcs $\alpha$ and $\beta$ in a surface $X$ are said to be in
\emph{general position} provided they intersect each other only transversely
and have no points of intersection of multiplicity greater than two
(other than possibly their end-points).
We denote this by writing $\alpha\pitchfork\beta$.
The \emph{geometric intersection number} of two arcs
$\alpha$ and $\beta$, denoted $I_X(\alpha,\beta)$ is defined by:
$$I_X(\alpha,\beta)=\min_{\substack{\alpha'\sim\alpha,\beta'\sim\beta \\
\alpha'\pitchfork\beta'}} |\alpha'\cap \beta'|,$$
where $\alpha'\cap\beta'$ denotes the intersection of $\alpha'$ and $\beta'$,
excluding their end-points. This is finite (see e.g.~\cite[\S3]{gp}).
We call elements of $\alpha'\cap\beta'$ \emph{crossings}.
We will usually just refer to $I_X(\alpha,\beta)$ as the \emph{intersection
number} of $\alpha,\beta$.
If $I_X(\alpha,\beta)=0$, we say that $\alpha$ and $\beta$ are \emph{noncrossing}.

If $X$ is oriented, a crossing between two oriented arcs $\alpha,\beta$
in $X$ in general position is called \emph{positive} if the orientation
arising from the pair of tangents to the
arcs at the crossing point is compatible with the orientation of $X$; it is
\emph{negative} otherwise. Figure~\ref{f:crossingtypes} illustrates the
two kinds of crossing for arcs in the plane (or $\U$). We remark that
the sign of a crossing in $\An$ is the same for either the orientation we
gave $\An$ in Section \ref{ssec:universalcover} (i.e.\ compatible with that
on $\Cn$) or the orientation inherited from its embedding into the plane.

\begin{figure}
\psfragscanon
\psfrag{X}{$\alpha$}
\psfrag{Y}{$\beta$}
\subfigure[Positive crossing]{\includegraphics[width=3cm]{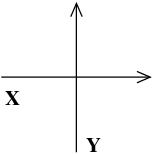}}
\hskip 1cm
\subfigure[Negative crossing]{\includegraphics[width=3cm]{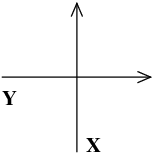}}
\caption{The two types of crossing between arcs $\alpha$ and $\beta$.}
\label{f:crossingtypes}
\end{figure}

We then define (for any pair of arcs in $X$)
$$I_X^+(\alpha,\beta)=\min_{\substack{\alpha'\sim\alpha,\beta'\sim\beta \\
\alpha'\pitchfork\beta'}} |\alpha'\cap^+ \beta'|,$$
where $\alpha'\cap^+\beta'$ denotes the set of positive crossings between
$\alpha'$ and $\beta'$.
The number $I_X^-(\alpha,\beta)$ is defined similarly.

\begin{remark} \label{rem:selfintersection}
If $\alpha=\beta$, then $I_X(\alpha,\alpha)$ is the minimum number
of crossings between a pair of arcs $\alpha',\beta'$ in general position
satisfying $\alpha'\sim\alpha$ and $\beta'\sim\alpha$.
For example, consider the arc $\alpha=\pi_8([0,9])$ in $\A(8)$.
Figure~\ref{fig:selfintersection} depicts two arcs $\alpha',\beta'$
homotopic to $\pi_8([0,9])$ in general position, with minimum number
of intersections. The intersections are indicated by small circles.
We have $I_{\A(8)}(\pi_8([0,9]),\pi_8([0,9]))=2$.
\end{remark}

\begin{figure}
\psfragscanon
\psfrag{0}{$0$}
\psfrag{1}{$1$}
\psfrag{2}{$2$}
\psfrag{3}{$3$}
\psfrag{4}{$4$}
\psfrag{5}{$5$}
\psfrag{6}{$6$}
\psfrag{7}{$7$}
\includegraphics[scale=.4]{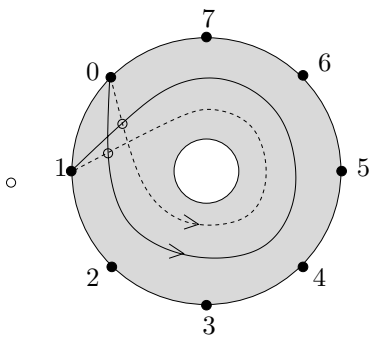}
\caption{$I_{\A(8)}(\pi_8([0,9]),\pi_8([0,9]))=2$.}
\label{fig:selfintersection}
\end{figure}

Recall that a \emph{lift} of an arc $\alpha$ in $\An$ (to $\U$)
is an arc $\widetilde{\alpha}$ in $\U$ such that $\pi_n(\widetilde{\alpha})=\alpha$.

\begin{lemma} \label{l:lift}
Let $\widetilde{\alpha},\widetilde{\beta}$ be lifts of oriented arcs
$\alpha,\beta$ in $\An$ respectively. Then
$$I_{\An}(\alpha,\beta)=\min_{\substack{\gamma\sim\widetilde{\alpha},\delta\sim\widetilde{\beta} \\
G\gamma\pitchfork G\delta}} \left|\gamma\cap G\delta \right|.$$
Furthermore, $I_{\An}^+(\alpha,\beta)$ (respectively, $I_{\An}^-(\alpha,\beta)$) can be computed
in the same way except that on the right hand side, only positive (respectively,
negative) crossings are counted.
\end{lemma}

\begin{proof} We note that the right hand side of the above is always finite
as $\gamma,\delta$ are compact.
We first prove the following claim: \\
\textbf{Claim}:
Let $\gamma$ and $\delta$ be lifts in $\U$ of $\alpha'$ and $\beta'$, respectively.
Then the following are equivalent.
\begin{enumerate}
\item[(a)] $\gamma\sim\widetilde{\alpha}$, $\delta\sim\widetilde{\beta}$, and
$G\gamma \pitchfork G\delta$.
\item[(b)] $\alpha'\sim\alpha$, $\beta'\sim\beta$, and $\alpha'\pitchfork\beta'$.
\end{enumerate}
\textit{Proof of claim:} (a) implies (b). Given $\gamma,\delta$ as in (a), the fact that
$\alpha'\sim\alpha$ and $\beta'\sim\beta$ follows from Theorem~\ref{t:monodromy}.
If $\alpha'$ and $\beta'$ intersected non-transversely, two lifts $\sigma^a(\gamma)$ and
$\sigma^b(\delta)$ of $\alpha'$ and $\beta'$ would also intersect non-transversely, a contradiction.
And if $\alpha'$ and $\beta'$ intersected in a point of multiplicity greater than $2$ then,
taking the pre-image of an admissible neighbourhood of that point, we'd have that
$G\gamma$ and $G\delta$ intersected in a point of multiplicity greater than $2$,
a contradiction. Hence (b) holds.

Conversely, suppose that $\alpha'$ and $\beta'$ are as in (b). Then
$\gamma\sim\widetilde{\alpha}$
and $\delta\sim\widetilde{\beta}$ follow from Theorem~\ref{t:monodromy}. If
$G\gamma$ and $G\delta$ intersected non-transversely,
the image of a non-transverse intersection would be a non-transverse intersection of
$\alpha'$ and $\beta'$, a contradiction. Similarly, if
$G\gamma$ and $G\delta$ intersected in a point of multiplicity greater than $2$,
the image would be a point of multiplicity greater than $2$ in the intersection of
$\alpha'$ and $\beta'$, a contradiction. Hence (a) holds.

Now in order to prove the lemma we must show that, in the circumstances of the claim,
we have
$$|\gamma\cap G\delta| = |\alpha'\cap \beta'|.$$
Let $f$ be $\pi_n$ restricted to $\gamma\cap G\delta$.
It is clear that the image of $f$ is contained in $\alpha'\cap \beta'$.
If $x\in\alpha'\cap \beta'$ there is $u\in \U$ such that $\pi_n(u)=x$.
Then we have:
$$u\in G\gamma \cap G\delta$$
since $\pi_n\gamma=\alpha'$ and $\pi_n\delta=\beta'$.
So $u\in \sigma^a(\gamma)\cap \sigma^b(\delta)$ for some $a,b\in\mathbb{Z}$.
Hence $\sigma^{-a}(u)\in \gamma\cap G\delta$ and
$\pi_n(\sigma^{-a}(u))=\pi_n(u)=x$ so $f$ is surjective.
If $u\not=u'\in \gamma \cap G\delta$ and
$f(u)=f(u')$ then $u,u'\in\gamma$, $u\in \sigma^a(\delta)$
and $u'\in \sigma^b(\delta)$ for some $a,b\in\mathbb{Z}$.
But then $\pi_n(u)$ is a point in the intersection of $\alpha'$
and $\beta'$ of multiplicity greater than $2$, a contradiction.
Hence $f$ is injective and thus a bijection.
We have the seen that the first part of the lemma holds. The other cases can be
dealt with in a similar way, since $\pi_n$ preserves orientation.
\end{proof}

To compute the intersection number of two admissible arcs
$\alpha$ and $\beta$ in $\An$, we will later exhibit explicit arcs in $\U$ homotopic to lifts
of $\alpha$ and $\beta$ giving an upper bound for $I_{\An}(\alpha,\beta)$
using Lemma~\ref{l:lift}. The following result gives a lower bound which
we will use to get equality. This lower bound is easily computable as it just involves
intersection numbers in $\U$.

\begin{corollary} \label{c:liftbound}
Let $\widetilde{\alpha},\widetilde{\beta}$ be lifts of oriented arcs
$\alpha,\beta$ in $\An$ respectively.
Then $$I_{\An}(\alpha,\beta)\geq \sum_{m\in\mathbb{Z}} I_{\U}(\widetilde{\alpha},\sigma^m(\widetilde{\beta})).$$
Similar results hold for $I_{\An}^+(\alpha,\beta)$ and $I_{\An}^-(\alpha,\beta)$.
\end{corollary}

\begin{proof}
Note that the sum on the right hand side is finite since $\widetilde{\alpha}$
and $\widetilde{\beta}$ are compact, so only finitely many of the intersection
numbers are non-zero. Using Lemma~\ref{l:lift}, we have:
\begin{eqnarray*}
I_{\An}(\alpha,\beta)
&=&\min_{\substack{\gamma\sim\widetilde{\alpha},\delta\sim\widetilde{\beta} \\
G\gamma\pitchfork G\delta}} \left|\gamma\cap G\delta \right| \\
&=&\min_{\substack{\gamma\sim\widetilde{\alpha},\delta\sim\widetilde{\beta} \\
G\gamma\pitchfork G\delta}} \sum_{m\in\mathbb{Z}} \left|\gamma\cap \sigma^m\delta \right| \\
&\geq&
\sum_{m\in\mathbb{Z}}
\min_{\substack{\gamma\sim\widetilde{\alpha},\delta\sim\widetilde{\beta} \\
G\gamma\pitchfork G\delta}} \left|\gamma\cap \sigma^m(\delta)\right| \\
&\geq&
\sum_{m\in\mathbb{Z}}
\min_{\substack{\gamma\sim\widetilde{\alpha},\delta\sim\sigma^m(\widetilde{\beta})\\
\gamma\pitchfork \delta}} \left|\gamma\cap \delta\right| \\
&=& \sum_{m\in\mathbb{Z}} I_{\U}(\widetilde{\alpha},\sigma^m(\beta)),
\end{eqnarray*}
as required.
\end{proof}

We can now put the above results together to obtain a computation of
intersection numbers of admissible arcs in the annulus.

\begin{prop} \label{p:intersectionnumbers}
Let $[a,b]$ and $[c,d]$ be admissible arcs in $\U$ and suppose that
$d-c\geq b-a$, i.e.\  the combinatorial length of $[c,d]$ is greater than or
equal to the combinatorial length of $[a,b]$. Then:
\begin{enumerate}
\item[(a)]
$$I_{\An}^+(\pi_n([a,b]),\pi_n([c,d]))=|\{m\in \mathbb{Z}\,:\,a<\sigma^m(c)<b\}|;$$
\item[(b)]
$$I_{\An}^-(\pi_n([a,b]),\pi_n([c,d]))=|\{m\in \mathbb{Z}\,:\,a<\sigma^m(d)<b\}|;$$
\item[(c)]
$$I_{\An}(\pi_n([a,b]),\pi_n([c,d]))=I_{\An}^+(\pi_n([a,b]),\pi_n([c,d]))+I_{\An}^-(\pi_n([a,b]),\pi_n([c,d])).$$
\end{enumerate}
\end{prop}

\begin{proof}
We note that $[c,d]$ is homotopy equivalent to the oriented arc $\delta$ consisting
of two line segments, the first, $L_1$, joining $(c,0)$ to $(\frac{1}{2}(c+d),h)$
and the second, $L_2$, joining $(\frac{1}{2}(c+d),h)$ to $(d,0)$, where
$0<h<1$ is fixed.
Note that in order to ensure the arc is smooth, we round it off at the point
$(\frac{1}{2}(c+d),h)$; it is clear that we can do this within a ball of radius
$\varepsilon>0$ for arbitrarily small $\varepsilon$ so it will not affect
the rest of the proof.

By translating $[a,b]$ if necessary, we may assume that
$c\leq a<c+n$. Let $M'_1$ be the line segment joining $(a,0)$ to
the point $(\frac{1}{2}(c+d+n),h)$, i.e.\ to the mid-point between the top of
$\delta$ and the top of $\sigma(\delta)$.

Let $\delta'$ be the $\sigma$-translate of $\delta$ whose end-point is $(y,0)$
with $y<b$ maximal. Let $M'_2$ be the line segment joining $(b,0)$ with
$(y+(c-d+n)/2,h)$, i.e.\ with the mid-point between the top of $\delta'$ and the
top of $\sigma(\delta')$.
Let $v$ be the point where $M'_1$ and $M'_2$ meet; then let
$M_1$ be the line joining $(a,0)$ to $v$ and let $M_2$ be the line joining
$v$ to $(b,0)$. Then $\gamma$, the oriented arc which starts at $(a,0)$, travels along
$M_1$ to $v$ and then along $M_2$ to $(b,0)$ is homotopy equivalent to $[a,b]$
(rounded off at $v$ as before).
See Figure~\ref{f:liftexample} for an example.

\begin{figure}
\psfragscanon
\psfrag{v}{$v$}
\psfrag{cc}{$\gamma$}
\psfrag{dd}{$\delta$}
\psfrag{c}{$c$}
\psfrag{a}{$a$}
\psfrag{b}{$b$}
\psfrag{d}{$d$}
\includegraphics[width=12cm]{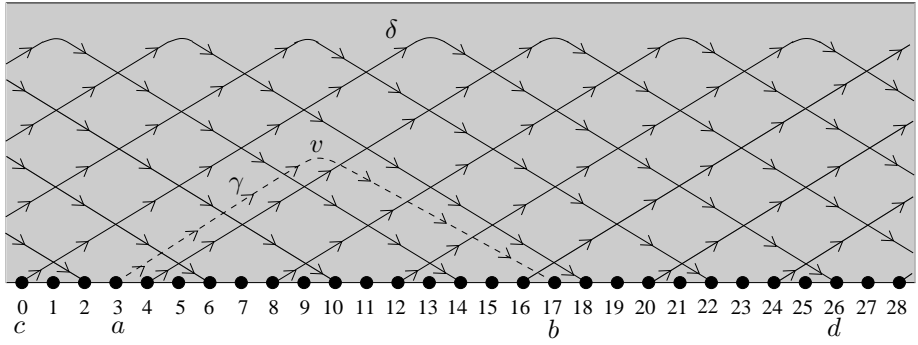}
\caption{Computing the intersection numbers of $\pi_4([0,26])$ and $\pi_4([3,17])$.
We have $I_{\A(4)}^+(\pi_4([0,26]),\pi_4([3,17]))=3$, $I_{\A(4)}^-(\pi_4([0,26]),\pi_4([3,17]))=4$ and $I_{\A(4)}(\pi_4([0,26]),\pi_4([3,17]))=7$.}
\label{f:liftexample}
\end{figure}

It is clear that
$\gamma$ intersects $G\delta$ transversely, and hence so does $G\gamma$.
Note that $M_1$ crosses only $\sigma$-translates of $L_2$ and
$\sigma$-translates of $M_2$.
Similarly, $M_2$ crosses only $\sigma$-translates of $L_1$ and $\sigma$-translates of $M_1$.
The same is true of $\sigma$-translates of $M_1$ and $M_2$.
It follows that there are no points of multiplicity greater than $2$
in the intersection of $G\gamma$ and $G\delta$.

Let $\alpha=\pi_n([a,b])$ and $\beta=\pi_n([c,d])$.
By Lemma~\ref{l:lift}, we have that $I_{\An}^+(\alpha,\beta)$ is
less than or equal to the number of positive crossings between $\gamma$ and
$G\delta$. Using the explicit description of $\gamma$ and $\delta$ above,
the latter number equals $\sum_{m\in\mathbb{Z}} I_{\U}^+(\gamma,\sigma^m\delta)$,
which equals $|\{m\in \mathbb{Z}\,:\,a<\sigma^m(c)<b\}|$.
By Corollary~\ref{c:liftbound}, we have
$$I_{\An}^+(\alpha,\beta)\geq \sum_{m\in\mathbb{Z}}I_{\U}^+([a,b],\sigma^m[c,d]).$$
Since $\gamma\sim [a,b]$ and $\delta\sim [c,d]$, it follows that
$$I_{\An}^+(\alpha,\beta)\geq \sum_{m\in\mathbb{Z}}I_{\U}^+(\gamma,\sigma^m\delta).$$
Putting this together with the above, we have equality and that part (a) of
the proposition holds.

Parts (b) and (c) of the proposition are proved in a similar way, using the
same representatives $\gamma$ and $\delta$ as above.
\end{proof}

For an example of Proposition~\ref{p:intersectionnumbers} in the case $n=4$, see
Figure~\ref{f:liftexample}.

We note that it follows from the above proof that the inequality in
Corollary~\ref{c:liftbound} is an equality, i.e.\ that intersection numbers
between admissible arcs in $\An$ can be computed in $\U$.

\begin{corollary} \label{c:intersectionnumberlift}
Let $\widetilde{\alpha},\widetilde{\beta}$ be lifts of admissible arcs
$\alpha,\beta$ in $\An$ respectively. Then
$$I_{\An}(\alpha,\beta)=\sum_{m\in \mathbb{Z}} I_{\U}(\widetilde{\alpha},\sigma^m(\widetilde{\beta})).$$
Similar results hold for $I_{\An}^+(\alpha,\beta)$ and $I_{\An}^-(\alpha,\beta)$.
\end{corollary}

\begin{remark}
Let $c$ be a single oriented cycle in $\An$ from $0$ to itself, i.e.\
a generator of the fundamental group at $0$. We note that it follows
from Proposition~\ref{p:intersectionnumbers} that, for $r,s>0$,
$I_{\An}^+(c^r,c^s)=I_{\An}^-(c^r,c^s)=\min(r,s)-1$ and $I_{\An}(c^r,c^s)=2\min(r,s)-2$.
\end{remark}

Recall the bijection between admissible arcs in $\An$ and
indecomposable objects in $\Tn$ from Lemma~\ref{l:ARiso}, taking
$\pi_n([a,b])$ to $M[a,b]$ for an admissible arc $\pi_n([a,b])$.
The object $M[a,b]$ has a unique composition series
(ordered from the socle upwards),
$S_{\overline{a+1}},S_{\overline{a+2}},\ldots ,S_{\overline{b-1}}$.
We now have all we need in order to show that the dimensions of extension groups (of degree $1$) in the tube can be interpreted as negative intersection numbers
of the corresponding admissible arcs in the annulus, our main result.

\begin{theorem} \label{t:extrule}
Let $\pi_n([a,b])$ and $\pi_n([c,d])$ be admissible arcs in $\An$.
Then
\begin{enumerate}
\item[(a)] $$\dim \Ext^1_{\Tn}(M[a,b],M[c,d])=I_{\An}^-(\pi_n([a,b]),\pi_n([c,d]));$$
\item[(b)] $$\dim \Ext^1_{\Tn}(M[c,d],M[a,b])=I_{\An}^+(\pi_n([a,b]),\pi_n([c,d]));$$
\item[(c)] $$\dim \Ext^1_{\CTn}(M[a,b],M[c,d])=I_{\An}(\pi_n([a,b]),\pi_n([c,d])).$$
\end{enumerate}
\end{theorem}

\begin{proof}
We first prove (a). We have
\begin{eqnarray*}
\Ext^1_{\Tn}(M[a,b],M[c,d])
&\cong&
D\Hom_{\Tn}(M[c,d],\tau M[a,b]) \\
&\cong&
D\Hom_{\Tn}(M[c,d],M[a-1,b-1])
\end{eqnarray*}
Assume first that the combinatorial length of $[c,d]$ is greater than or equal to that of
$[a,b]$. Then a homomorphism in $\Hom_{\Tn}(M[c,d],M[a-1,b-1])$
must map the top,
$S_{\overline{d-1}}$, of $M[c,d]$ onto a composition factor $S_{\overline{d-1}}$ in $M[a-1,b-1]$ and is uniquely determined up to a scalar by this choice.
So the dimension of the last space above is the number of times
$\overline{d-1}$ appears in the sequence $[\overline{a},\overline{a+1},\ldots ,\overline{b-2}]$, i.e.\ the
number of times $\overline{d}$ appears in the sequence $\overline{a+1},\overline{a+2},\ldots ,\overline{b-1}$.
This is equal to $I_{\An}^-(\pi_n([a,b]),\pi_n([c,d]))$ by
Proposition~\ref{p:intersectionnumbers}.

If the combinatorial length of $[c,d]$ is less than or equal to that of $[a,b]$,
a homomorphism in $\Hom_{\Tn}(M[c,d],M[a-1,b-1]$
must map a composition factor $S_{\overline{a}}$ of $M[c,d]$ onto the socle
$S_{\overline{a}}$ of $M[a-1,b-1]$,
and is uniquely determined up to a scalar by this choice. Hence the dimension
of this $\Hom$-space is the number of times $\overline{a}$ appears in the
sequence $[\overline{c+1},\overline{c+2},\ldots ,\overline{d-1}]$, i.e.\ $I_{\An}^+(\pi_n([c,d]),\pi_n([a,b]))$ by
Proposition~\ref{p:intersectionnumbers}. This is the same as
$I_{\An}^-(\pi_n([a,b]),\pi_n([c,d]))$, as required. Part (a) is proved.

The proof of part (b) is similar. Part (c) follows from the fact that:
\begin{equation} \label{e:extsymmetrising}
\Ext_{\CTn}(M[a,b],M[c,d])\cong \Ext_{\Tn}(M[a,b],M[c,d])\oplus D\Ext_{\Tn}(M[c,d],M[a,b]),
\end{equation}
(proved as in~\cite[Prop. 1.7]{bmrrt}) where $D=\Hom_k(-,k)$ is the linear duality, and Proposition~\ref{p:intersectionnumbers}(c).
\end{proof}

\begin{remark}
In the light of Theorem~\ref{t:extrule},
Proposition~\ref{p:intersectionnumbers}(c) can be regarded as a geometric
version of equation~\eqref{e:extsymmetrising} for the cluster tube, i.e.\
passing to the cluster category (of the tube) corresponds to passing from
oriented to unoriented arcs in the annulus.
\end{remark}

Using the results~\cite[5.1,5.3]{bz}, (in the case of an annulus $X$ with a
single marked point on the inner boundary and $n$ marked points on the outer
boundary), we get that, in
the situation of Theorem~\ref{t:extrule}, $I_{\An}(\pi_n([a,b]),\pi_n([c,d]))\neq 0$ implies that
$\Ext^1_{\CTn}(M[a,b],M[c,d])\neq 0$ and $\Ext^1_{\CTn}(M[c,d],M[a,b])\neq 0$
(with equivalence in the case where $\pi_n([a,b])=\pi_n([c,d])$). We use the
fact (see~\cite[\S2]{bmv}) that the cluster tube is a thick subcategory of the cluster category associated to $X$. Thus Theorem~\ref{t:extrule}(c) can be
regarded as a strengthening of this result in the case of a tube. See also~\cite[4.23]{w} for an alternative proof of parts (a) and (b).

We also remark that a geometric interpretation of dimensions of
extension groups in cluster categories (i.e.\ a result analogous to
Theorem~\ref{t:extrule}(c)) has been obtained in type A~\cite[Rk.\ 5.3]{ccs},
in type D~\cite[Thm.\ 4.3]{sc} and in type $A_{\infty}$~\cite[Lemma 3.6]{hj} (see also
Section~\ref{s:Ainfinity}). Note that in the tube or cluster tube, and also in the setting
of~\cite{bz}, but not in~\cite{ccs,sc,hj}, extension groups between indecomposable objects
can have arbitrarily large dimension; in fact the dimension of the extension
group of an indecomposable object with itself can be arbitrarily large. Furthermore, all but finitely
many objects in the tube (or cluster tube) have non-zero self-extension groups,
corresponding to the fact that the corresponding oriented (respectively, unoriented) arcs
have non-zero self-intersection numbers.

Recall that an object $M$ in $\Tn$ is said to be \emph{rigid} if $\Ext_{\Tn}^1(M,M)=0$, and
\emph{maximal rigid} if $\Ext^1_{\Tn}(M\oplus X,M\oplus X)=0$ implies $X$ lies in $\add M$ for any
object $X$ in $\Tn$. By Theorem~\ref{t:extrule}(c), we see that an object of $\Tn$ is rigid if
and only if the admissible arcs corresponding to its indecomposable summands do not cross (themselves or each
other), and such an object is maximal rigid if and only if it is a maximal noncrossing collection of
admissible arcs.
We recall that in~\cite[\S2]{bu}, a general notion of triangulation is considered for marked surfaces
(such as $\An$) in which it is allowed for boundary components not to contain marked points.
The arcs in the triangulation, which must not be self-crossing, must divide up the surface into triangles (possibly self-folded) or annuli with a single marked point on the outer boundary.

\begin{prop}
The map $\varphi$ from admissible arcs in $\An$ to $\ind(\Tn)$ induces a bijection between
triangulations of $\An$ (in the sense of~\cite{bu}) and maximal rigid objects of $\ind(\Tn)$.
\end{prop}

\begin{proof}
It is clear that any triangulation of $\An$ must contain an arc dividing $\An$ into a disk
and one such annulus (since at least one component of the complement of the arcs in the triangulation
must have the inner boundary component of $\An$ on its boundary).
The remaining arcs in the triangulation triangulate the disk. Forgetting the arcs in the triangulation
homotopic to arcs along the boundary and orienting all arcs anticlockwise, it is clear we obtain a
maximal collection of noncrossing admissible arcs.

Conversely, any maximal collection of noncrossing admissible arcs must contain the arc from some
marked point on the outer boundary to itself, looping once around the inner boundary component (if it
doesn't, this arc can be added). As above, this divides $\An$ into a disk and an annulus of the above
kind. The remaining arcs form a maximal noncrossing collection of admissible arcs in the disk, and thus
triangulate it. It follows that, if we ignore the orientations and add all the boundary arcs (arcs between
adjacent marked points on the outer boundary, and the arc around the inner boundary), we obtain a triangulation.

It is clear that the above maps are inverses to each other. It follows that we have a bijection between triangulations of $\An$ and maximal noncrossing collections of admissible arcs and we are done.
\end{proof}

\section{Doubled arc model}

In~\cite[\S3]{bmv}, the additive subcategory of a tube of rank $n$ generated by
the rigid indecomposable objects is modelled by pairs of unoriented arcs in a regular $2n$-sided polygon.
We now indicate how this model is related to the model presented here.

For an integer $n\geq 2$, consider the annulus $\An$ as being embedded
in the complex plane with its centre at the origin with marked points
(on the outer boundary) at the $n$th roots of unity.
Then the map $\psi:z\mapsto z^2$ (followed by a scaling to map the new
inner boundary onto the same circle it was to start with)
induces a map from $\A(2n)$ to $\An$, mapping the
marked points of $\A(2n)$ onto those of $\An$. We see that
$\psi$ maps $i,i+n$ onto $i$ for each $i$. The preimage under $\psi$ of
$\pi_n([a,b])$ is $\pi_{2n}([a,b])\cup \pi_{2n}([a+n,b+n])$, up to homotopy
equivalence. This pair of oriented arcs corresponds to a pair of oriented arcs
in a regular $2n$-gon as in~\cite[\S3]{bmv} (if the orientation is dropped), except that here the diameters are represented as a pair.

\begin{prop} \label{p:doubledarcs}
Let $[a,b]$, $[c,d]$ be admissible arcs in $\U$. Then
\begin{enumerate}
\item[(a)]
$$I_{\A(2n)}^+(\psi^{-1}(\pi_n([a,b]),\psi^{-1}(\pi_n([c,d]))) =
2I_{\An}^+(\pi_n([a,b]),\pi_n([c,d])).$$
\item[(b)]
$$I_{\A(2n)}^-(\psi^{-1}(\pi_n([a,b]),\psi^{-1}(\pi_n([c,d]))) =
2I_{\An}^-(\pi_n([a,b]),\pi_n([c,d])).$$
\item[(c)]
$$I_{\A(2n)}(\psi^{-1}(\pi_n([a,b]),\psi^{-1}(\pi_n([c,d]))) =
2I_{\An}(\pi_n([a,b]),\pi_n([c,d])).$$
\end{enumerate}
\end{prop}

\begin{proof}
We note that for oriented arcs $\gamma,\gamma',\delta$ in $\U$ which do not self-intersect,
it is easy to see that
$$I_{\U}^+(\gamma\cup\gamma',\delta)=I_{\U}^+(\gamma,\delta)+I_{\U}^+(\gamma',\delta).$$

It follows from
Corollary~\ref{c:intersectionnumberlift} that $I^+_{\An}$ has the same
additivity property.
Let $[a,b]$, $[c,d]$ be admissible arcs in $\U$. Then:
\begin{equation*}
\begin{split}
I_{\A(2n)}^+(\psi^{-1}(\pi_n([a,b]),\psi^{-1}(\pi_n([c,d]))) &=
I_{\A(2n)}^+(\pi_{2n}([a,b]),\pi_{2n}([c,d])) \\
&\quad +I_{\A(2n)}^+(\pi_{2n}([a,b]),\pi_{2n}([c+n,d+n])) \\
&\quad +I_{\A(2n)}^+(\pi_{2n}([a+n,b+n]),\pi_{2n}([c,d])) \\
&\quad +I_{\A(2n)}^+(\pi_{2n}([a+n,b+n]),\pi_{2n}([c+n,d+n])).
\end{split}
\end{equation*}

Hence
\begin{equation*}
\begin{split}
I_{\A(2n)}^+(\psi^{-1}(\pi_n([a,b]),\psi^{-1}(\pi_n([c,d]))) &=
|\{m\in\mathbb{Z}\,:\,a<\sigma^{2m}(c)<b\}| \\
&\quad +|\{m\in\mathbb{Z}\,:\,a<\sigma^{2m+1}(c)<b\}| \\
&\quad +|\{m\in\mathbb{Z}\,:\,\sigma(a)<\sigma^{2m}(c)<\sigma(b)\}| \\
&\quad +|\{m\in\mathbb{Z}\,:\,\sigma(a)<\sigma^{2m+1}(c)<\sigma(b)\}|,
\end{split}
\end{equation*}
so
\begin{eqnarray*}
I_{\A(2n)}^+(\psi^{-1}(\pi_n([a,b]),\psi^{-1}(\pi_n([c,d]))) &=&
2I_{\An}^+(\pi_n([a,b]),\pi_n([c,d])),
\end{eqnarray*}
as required for (a). Part (b) is proved similarly, and part (c) follows.
\end{proof}

\begin{corollary}
Let $[a,b]$, $[c,d]$ be admissible arcs in $\U$.
Then
$$2\dim\Ext^1_{\CTn}(M[a,b],M[c,d])=I_{\A(2n)}(\psi^{-1}(\pi_n([a,b])),\psi^{-1}(\pi_n([c,d]))).$$
\end{corollary}

\begin{proof}
This follows from Proposition~\ref{p:doubledarcs}(c) and Theorem~\ref{t:extrule}.
\end{proof}

We thus recover~\cite[3.2]{bmv} in the case where $\pi_n([a,b])$ and
$\pi_n([c,d])$ are non self-intersecting, i.e.\ the corresponding indecomposable objects
in $\Tn$ are rigid.

\section{AR-quivers of type $\mathbb{Z}A_{\infty}$} \label{s:Ainfinity}

By~\cite[III.1.1]{rv}, the category $\Tinf$ of finite
dimensional representations of the quiver $A_{\infty}^{\infty}$:

$$\xymatrix{
\cdots -3 \ar[r] & -2 \ar[r] & -1 \ar[r] & 0 \ar[r] & 1 \ar[r] & 2 \ar[r] & 3 \cdots
}$$
is a connected Hom-finite noetherian hereditary abelian category with almost split
sequences whose AR-quiver is of type $\mathbb{Z}A_{\infty}$. Hence one can form
the cluster category $\CTinf$ of $\Tinf$ as the quotient
$D^b(\Tinf)/(\tau^{-1}[1])$; see~\cite[2.1]{kr1}.
Since $\Tinf$ is skeletally small, it follows from~\cite[9.9]{k1}
that it is triangulated and thus also $2$-Calabi-Yau.
As pointed out to us by P. Jorgensen, it follows from~\cite[Thm. 2.1]{kr2}
that this category coincides with the category $D$ considered by Holm and
Jorgensen in~\cite{hj} (referred to there as a cluster category of type
$A_{\infty}$).

The indecomposable objects in $\Tinf$ (and thus, also, in $\CTinf$)
are, up to isomorphism, of the form
$X[a,b]$ where $[a,b]$ is an admissible arc in $\U$. The representation
$X[a,b]$ has a unique composition series (ordered from the socle upwards)
given by $S_{a+1},S_{a+2},\ldots ,S_{b-1}$. There is
an irreducible map in $\Tinf$ from $X[a,b]$ to $X[c,d]$, for admissible arcs
$[a,b]$ and $[c,d]$ in $\U$, if and only if $c=a$ and $d=b+1$ or $c=a+1$ and $d=b$ (compare with Subsection~\ref{ss:quiverofarcs}). The AR-translate takes
$X[a,b]$ to $X[a-1,b-1]$. The irreducible maps and AR-translate have the same
description in $\Tinf$ and in $\CTinf$; see~\cite[1.3]{bmrrt}.

\begin{theorem} \label{t:extruleinfinity}
Let $[a,b]$ and $[c,d]$ be admissible arcs in $U$. Then
\begin{enumerate}
\item[(a)] $$\dim \Ext^1_{\Tinf}(X[a,b],X[c,d])=I_{\U}^-([a,b],[c,d]);$$
\item[(b)] $$\dim \Ext^1_{\Tinf}(X[c,d],X[a,b])=I_{\U}^+([a,b],[c,d]);$$
\item[(c)] $$\dim \Ext^1_{\CTinf}(X[a,b],X[c,d])=I_{\U}([a,b],[c,d]).$$
\end{enumerate}
\end{theorem}

\begin{proof}
This result is proved in the same way as Theorem~\ref{t:extrule} (except that
the reduction modulo $n$ is omitted). Note that the same proof of
formula~\eqref{e:extsymmetrising} holds in this context.
\end{proof}

\noindent \textbf{Acknowledgements:}
Karin Baur was supported by the Swiss National Science Foundation (grant PP0022-114794).
This work was supported by the Engineering and Physical Sciences Research Council
[grant numbers EP/S35387/01 and EP/G007497/1] and the
Institute for Mathematical Research (FIM, Forschungsinstitut f\"ur
Mathematik) at the ETH, Z\"{u}rich.

We would like to thank Bernhard Keller and Matthias \linebreak
Warkentin for their
helpful comments. Robert Marsh would like to thank Karin Baur and the Institute for
Mathematical Research (FIM) at the
ETH, Zurich, for their kind hospitality during visits in June 2009 and
August-September 2010. We would also like to thank the referees for their
helpful comments on an earlier version of this paper.

\end{document}